\documentclass[11pt]{article}
\usepackage{amsmath, amssymb, amsfonts, amsthm, epsfig, float, graphicx, sectsty}
\usepackage[all]{xy}

\title{Short Laws for Finite Groups  
and Residual Finiteness Growth}
\author{Henry Bradford and Andreas Thom}
\date{}

\newcommand{\Addresses}{{
  \bigskip
  \footnotesize

H. Bradford, \textsc{Georg-August-Universit\"at G\"ottingen, Germany}\par\nopagebreak
  \textit{E-mail address:} \texttt{henry.bradford@mathematik.uni-goettingen.de}

  \medskip

A. Thom, \textsc{TU Dresden, Germany}\par\nopagebreak
  \textit{E-mail address:} \texttt{andreas.thom@tu-dresden.de}

}}

\newtheorem{thm}{Theorem}[section]
\newtheorem{lem}[thm]{Lemma}
\newtheorem{propn}[thm]{Proposition}
\newtheorem{coroll}[thm]{Corollary}
\newtheorem{defn}[thm]{Definition}

\newtheorem{ex}[thm]{Example}

\newtheorem{qu}[thm]{Question}
\newtheorem{prob}[thm]{Problem}

\DeclareMathOperator{\diam}{diam}

\DeclareMathOperator{\im}{im}

\DeclareMathOperator{\Alt}{Alt}
\DeclareMathOperator{\Aut}{Aut}
\DeclareMathOperator{\GL}{GL}

\DeclareMathOperator{\PSL}{PSL}

\DeclareMathOperator{\PSU}{PSU}
\DeclareMathOperator{\SL}{SL}

\DeclareMathOperator{\Sym}{Sym}

\begin{document}

\maketitle

\begin{abstract}
We prove that for every $n \in \mathbb{N}$ and $\delta>0$ 
there exists a word $w_n \in F_2$ of length 
$O(n^{2/3} \log (n)^{3+\delta})$ which is a law for every finite group 
of order at most $n$. 
This improves upon the main result of \cite{Thom} 
by the second named author. 
As an application we prove a new lower bound on the residual 
finiteness growth of non-abelian free groups. 
\end{abstract}

\section{Introduction}

A \emph{law} for a group $G$ 
is an equation which holds identically in $G$. 
The qualitative study of laws in groups is a classical subject, 
often phrased in terms of \emph{varieties of groups,} \cite{BNeum, HNeum}. Moreover certain specific laws have been the subject 
of intense study over the years, particularly power laws, 
which delineate the territory of the 
\emph{bounded} and \emph{restricted Burnside problems}. 
A more recently opened avenue of research has been the asymptotic 
behaviour of lengths of laws for sequences of finite groups, 
motivated by connections with other 
invariants of interest in asymptotic group theory. 

\subsection{Statement of Results}

Our focus in this paper will be on the length of laws 
which are satisfied \emph{simultaneously} 
by all sufficiently small finite groups. 
Our main result is as follows. 

\begin{thm} \label{allgrpsthm}
For all $n \in \mathbb{N}$ 
there exists a word $w_n \in F_2$ of length
\begin{center}
$O\left(n^{2/3} \log (n)^3 \log^*(n)^{2}\right)$
\end{center}
such that for every finite group $G$ satisfying $\lvert G \rvert \leq n$, 
$w_n$ is a law for $G$. 
\end{thm}

In the statement of the previous theorem, $\log^*(n)$ denotes the iterated logarithm, i.e.\ the smallest natural number $k$, such that the $k$-fold application of $\log \colon (0,\infty) \to (-\infty,\infty)$ to $n$ yields a real number less than $1$. Note that $\log^*(n)$ grows slower than any iteration of logarithms.

The precise definition of a \emph{law} for a group 
is given in Subsection \ref{lawssubsect}. 
Theorem \ref{allgrpsthm} improves asymptotically on the result of \cite{Thom}, 
in which the second named author obtained an upper bound of 
$$O \left(n \log \log (n)^{9/2} / \log (n)^2 \right).$$ 

Our second result, 
which is a direct application of Theorem \ref{allgrpsthm}, 
is a new lower bound on the residual finiteness growth of nonabelian free groups. 
Recall that a group $\Gamma$ is \emph{residually finite} if, for every $1_{\Gamma} \neq g \in \Gamma$, 
there exists a finite group $H$ and a homomorphism $\pi \colon \Gamma \rightarrow H$ such that $\pi (g) \neq 1_H$. 
Bou-Rabee \cite{BouRab} introduced a quantitative version of this property, as follows. 
Given a residually finite group $\Gamma$ and $1_{\Gamma} \neq g \in \Gamma$, we define:
\begin{center}
$k_{\Gamma} (g) = \min \lbrace \lvert H \rvert  \mid \text{there exists }\pi \colon \Gamma \rightarrow H, \pi (g) \neq 1_H \rbrace$. 
\end{center}
Now, given a generating set $S$ for $\Gamma$, there is a naturally associated length function on $\Gamma$. 
For $n \in \mathbb{N}$, let $B_S (n) \subseteq \Gamma$ 
be the set of elements of length at most $n$. We define: 
\begin{center}
$\mathcal{F}_{\Gamma} ^S (n) = \max \lbrace k_{\Gamma} (g) \mid 1_{\Gamma} \neq g \in \Gamma, g \in B_S (n) \rbrace$. 
\end{center}
Informally, we may think of groups $\Gamma$ for which $\mathcal{F}_{\Gamma} ^S (n)$ grows slowly as being those in which non-trivial elements are ``easy'' to detect in finite quotients. 
Now, $\mathcal{F}_{\Gamma} ^S$, being defined in terms of the generating set $S$, 
is not an invariant of $\Gamma$ alone. However, if $\Gamma$ is finitely generated, 
then $\mathcal{F}_{\Gamma} ^S$ turns out to be independent of $S$ up to the equivalence relation induced by a natural 
partial order $\preceq$ on functions, which we make precise below. 
Thus we can speak without ambiguity about the \emph{residual finiteness growth} of $\Gamma$. 
Since the introduction of this notion, particular attention has been paid to the task of estimating 
$\mathcal{F}_{\Gamma} ^S$ for the nonabelian free groups $F_k$ ($k \geq 2$), 
starting with Bou-Rabee's original paper \cite{BouRab}, 
and continuing through a series of papers by various authors \cite{BouMcR, KasMat, Thom}. 
Since $F_k$ has a very rich family of finite quotients, 
one expects its residual finiteness growth to be very slow. 
Consequently, any significant lower bounds on $\mathcal{F}_{F_k} ^S$
represent a surprising group-theoretic phenomenon. 
There is a clear connection between such lower bounds 
and laws for finite groups: 
a non-trivial element $w \in F_k$ of length $n$ which is a law for all finite groups of order at most $f(n)$ 
witnesses that $\mathcal{F}_{F_k} ^S (n) > f(n)$. 
From this observation, Theorem \ref{allgrpsthm} immediately implies the following new lower bound on the residual 
finiteness growth of $F_k$. 

\begin{thm} \label{RFGthm}
$\mathcal{F}_{F_k} ^S (n) \succeq n^{3/2} / \log (n)^{9/2+\varepsilon}$. 
\end{thm}

It is likely that the conclusion of Theorem \ref{RFGthm} is best possible up to logarithmic factors: 
following on from work of Hadad \cite{Had}, 
Kassabov and Matucci (\cite{KasMat}, Remark 9) propose that the shortest non-trivial law 
satisfied simultaneously by the groups $SL_2 (R)$, 
as $R$ varies over all finite commutative rings with $\lvert R \rvert \leq N$, 
should be of length $\Omega (N^2)$. They further note that, 
if this conjecture is correct, then $\mathcal{F}_{F_k} ^S (n) \preceq n^{3/2}$. 

\subsection{Background} \label{backgroundsubsect}

The first result on residual finiteness growth of free groups 
appeared in Bou-Rabee's original paper introducing the notion
(see also Rivin \cite{Riv}). 

\begin{thm}[\cite{BouRab}] \label{BouRabThm}
$\mathcal{F}_{F_k} ^S (n) \preceq n^3$. 
\end{thm}

This result is an immediate consequence of a corresponding theorem 
on the residual finiteness growth of linear groups, 
via the embedding of $F_k$ into $\SL_2 (\mathbb{Z})$. 
Theorem \ref{BouRabThm} implies that a law holding 
simultaneously in all finite groups of order at most $n$ 
must have length $\Omega (n^{1/3})$. 
To date these are the best complementary bounds to 
Theorem \ref{allgrpsthm} and \ref{RFGthm}. 

The first attempt to investigate the problem addressed 
by Theorem \ref{allgrpsthm} was made by Bou-Rabee and McReynolds, 
though their result is again phrased in terms of residual finiteness growth. 

\begin{thm}[\cite{BouMcR}] \label{BouMcRThm}
$\mathcal{F}_{F_k} ^S (n) \succeq n^{1/3}$. 
\end{thm}

Theorem \ref{BouMcRThm} implies that there exists 
a word $w_n$ of length $O (n^3)$ 
satisfying the conclusion of Theorem \ref{allgrpsthm}. 
Note that this already improves dramatically 
over the obvious laws $w_n = x^{n !}$ and say $w'_n = x^{{\rm lcm}(1,\dots,n)}$ that merely show $\mathcal{F}_{F_k} ^S (n) \succeq \log(n)$. 
Bou-Rabee and McReynolds' construction was refined 
by Kassabov and Matucci, 
who obtained the following. 

\begin{thm}[\cite{KasMat}] \label{KasMatThm}
For all $n \in \mathbb{N}$, there exists a word $w_n \in F_2$ of length $O (n^{3/2})$
such that for every finite group $G$ satisfying 
$\lvert G \rvert \leq n$, $w_n$ is a law for $G$. 
Consequently $\mathcal{F}_{F_k} ^S (n) \succeq n^{2/3}$. 
\end{thm}

As discussed above, prior to the present paper 
the best upper bound for the lengths of the $w_n$ 
was the following result of the second author. 

\begin{thm}[\cite{Thom}] \label{Thomallgrpsthm}
For all $n \in \mathbb{N}$, there exists a word $w_n \in F_2$ of length
\begin{center}
$O (n \log \log (n)^{9/2} / \log (n)^2)$
\end{center}
such that for every finite group $G$ satisfying $\lvert G \rvert \leq n$, 
$w_n$ is a law for $G$. 
\end{thm}

The proof of this last result differs significantly 
from those which had preceded it, 
in that it makes extensive use of deep results from finite group 
theory, including the Classification of Finite Simple Groups. 

It was a great surprise to the authors 
to discover that the main term $n$ 
appearing in the bound from Theorem \ref{Thomallgrpsthm} 
was not best possible. 
Indeed it was conjectured in \cite{Thom} that it should be. 
In a related vein, Kassabov and Matucci asked the following. 

\begin{qu}[\cite{KasMat}] \label{KasMatqu}
For $k \geq 2$, is it the case that 
$\mathcal{F}_{F_k} ^S (n) \simeq n$? 
\end{qu}

Although Theorem \ref{Thomallgrpsthm} 
already implies a negative answer 
to Question \ref{KasMatqu} 
(by the observations from our Subsection \ref{RFGsubsect}, 
it follows from Theorem \ref{Thomallgrpsthm} that
$\mathcal{F}_{F_k} ^S (n) \succeq n \log (n)^{2-\epsilon}$), 
until our work it remained eminently plausible that 
linear bounds for $\mathcal{F}_{F_k} ^S (n)$ and 
the length of the shortest words $w_n$ 
satisfying the conclusion of Theorems \ref{allgrpsthm} and \ref{Thomallgrpsthm} could not be beaten by more than a polylogarithmic factor. 
The fact that polynomial improvements 
could be achieved was quite unexpected. 

As regards individual groups, 
much recent attention has been devoted to the length of laws 
for simple groups. 
For the symmetric (and therefore also alternating) groups, 
the best known upper bound is provided by a result of Kozma and the 
second author. 

\begin{thm}[\cite{KozTho}]  \label{LawsforAlt}
There exists a law for $\Sym(n)$ of length at most: 
\begin{center}
$\exp (O(\log (n)^4 \log \log (n)))$. 
\end{center}
\end{thm}

If $G$ is a finite group of Lie type, 
then the best available bound appears in another paper of the authors. 

\begin{thm}[\cite{BraTho(Lie)}] \label{Liethm}
Let $G$ be a finite group of Lie type over a field of order $q$, 
such that the natural module for $G$ has dimension $d$. 
Then there is a word $w_G \in F_2$ of length: 
\begin{center}
$q^{\lfloor d/2 \rfloor} \log (q)^{O_d (1)}$
\end{center}
which is a law for $G$. 
\end{thm}

Furthermore, the exponent $\lfloor d/2 \rfloor$ of $q$ in 
Theorem \ref{Liethm} is known to be sharp: 
Hadad \cite{Had} gives a lower bound of 
$\Omega (q^{\lfloor d/2 \rfloor})$ 
for the length of the shortest law for $\SL_d (q)$, using an embedding of $\SL_2(q^{\lfloor d/2 \rfloor})$. Thus, the worst case is $\SL_2$ -- a reoccuring theme in finite group theory. The same phenomenon will appear in our study of laws for all groups up to size $n$.

We will employ Theorem \ref{Liethm} 
in the proof of Theorem \ref{allgrpsthm}. 
As such a few words on the proof of Theorem \ref{Liethm} are in order. 
The proof strategy was inspired by that of Theorem \ref{LawsforAlt}: 
in both cases the problem is first divided 
into a search for two words vanishing, 
respectively, on \emph{generating} and \emph{non-generating} pairs 
of elements in our group $G$. 
The constructions of these two words are quite different. 
For generating pairs, we identify a large subset $E \subseteq G$  satisfying a short identity. For $G = \Sym (n)$, the set $E$ is the set of $n$-cycles; 
for $G$ of Lie type it is usually the union of split maximal tori. 
We then show that many short words in our generating pair of elements 
will lie in $E$, using results on mixing times and random walks in $G$ 
(see subsection \ref{diamsubsect}) and compose these words with the identity holding in $E$.  

A non-generating pair of elements lie in a common maximal subgroup 
$M$ of $G$. By classifying the possibilities for $M$ 
we can produce laws by induction. 
For $G = \Sym (n)$, this is facilitated by the O'Nan-Scott Theorem, 
and consequences thereof due to Liebeck 
(see \cite{KozTho} for details). 
To carry out the induction needed to prove Theorem \ref{Liethm} 
in full generality is a deep matter, 
requiring Aschbacher's Theorem on maximal subgroups of finite classical groups, the Classification of Finite Simple Groups, results on the classification of maximal subgroups in 
exceptional groups of Lie type, and dimension bounds for 
permutational and linear representations of finite simple groups. 
That said, to prove Theorem \ref{allgrpsthm} 
we will only apply Theorem \ref{Liethm} in the case $G = \PSL_3 (q)$ 
or $\PSU_3 (q)$, for which the determination of maximal subgroups 
is classical (see \cite{King}). 

One common feature of the proofs of Theorems \ref{allgrpsthm}, 
\ref{LawsforAlt} and \ref{Liethm} is the use of probabalistic arguments to demonstrate the existence of words with certain properties. That is, it is shown that, 
according to a certain model of random words, 
a word has the desired property with positive probability, 
and it is deduced that such a word must exist. 
As such, the proofs are non-constructive: 
they do not facilitate the description of explicit laws for the given groups. By contrast, Theorem \ref{Thomallgrpsthm} 
is constructive, and words $w_n$ satisfying 
the conclusions of that theorem could in principle 
be explicitly written down. 
It remains a beguiling question what the form of the words $w_n$ 
arising in Theorem \ref{allgrpsthm} might be. 

\subsection{Outline of the Paper}

In Section \ref{prelimsect} we assemble some basic tools for constructing laws in groups (Subsection \ref{lawssubsect}), 
material on mixing times and random walks in finite groups 
(Subsection \ref{diamsubsect}), 
and notions relating to residual finiteness growth, 
including the deduction of Theorem \ref{RFGthm} 
from Theorem \ref{allgrpsthm} (Subsection \ref{RFGsubsect}). 

The proof of Theorem \ref{allgrpsthm}, 
to which Section \ref{mainproofsect} is devoted, 
has much in common 
with that of Theorem \ref{Thomallgrpsthm} found in \cite{Thom}. 
There, the problem of constructing a short law valid 
in \emph{all} sufficiently small finite groups 
was first reduced to that of constructing a short law valid only in all sufficiently small \emph{simple} groups. 
This reduction will also be the first step in the proof of Theorem 
\ref{allgrpsthm}. It is summarised in Subsection \ref{redsimplesubsect}, and appears in full detail in \cite{Thom}. 
Further standard reductions 
(discussed in Subsection \ref{lawssubsect}) 
allow us to consider separately each of the eighteen infinite families 
of finite simple groups (since the bound in Theorem \ref{allgrpsthm} is asymptotic, the sporadic groups are easily eliminated). 

The key novelty of our work lies in our treatment of groups 
of the form $\PSL_2 (q)$, $\PSL_3 (q)$ and $\PSU_3 (q)$ 
(for all other finite simple groups, the laws 
constructed in \cite{Thom} suffice; 
this is explained in Subsection \ref{lowrkreductsubsect}). 
Once again, we may consider these three classes separately 
from the other finite simple groups, and from each other. 
$\PSL_3 (q)$ and $\PSU_3 (q)$ are then easily dealt with using 
Theorem \ref{Liethm}: we have a short law for each individual group 
and these can be combined over all sufficiently small $q$. 
This is carried out in Subsection \ref{PSL3subsect}. 

By contrast, combining laws for individual groups of the form 
$\PSL_2 (q)$ into one law valid for all of them is too expensive 
(essentially, there are too many such groups of small order). 
This issue provided the bottleneck for the bound in Theorem \ref{Thomallgrpsthm}. Instead, we take an approach closer in spirit
to that used in the \emph{proof} of Theorem \ref{Liethm}, 
as it was sketched in Subsection \ref{backgroundsubsect}. 
Namely, given a generating pair in some $\PSL_2 (q)$, 
we use random walks to locate many short words potentially satisfying a 
short relation. Those words are then combined using an iterated commutator to yield one word that works with high probability. 

In contrast to previous work, however, 
we run random walks in pairs, and locate pairs of elements 
lying in a common Borel subgroup. 
Since there is a common law satisfied by the Borel subgroups 
of \emph{every} $\PSL_2 (q)$ (a double commutator), 
there is no need to combine laws for the individual groups: 
the words we produce using the random walk method 
will already provide laws holding simultaneously in $\PSL_2 (q)$ 
for all sufficiently small $q$. This approach works as stated for a set of good primes arising from the work of Breuillard and Gamburd \cite{BreGam} on uniform expansion for $\PSL_2(p)$. This argument, the subject of Subsection \ref{PSL2subsect}, 
is the technical heart of our work. The groups $\PSL_2(q)$ for other primes, prime powers and other finite simple groups are dealt with by a more direct argument.  
The paper concludes with a short survey of open problems. 

\section{Preliminaries} \label{prelimsect}

\subsection{Laws in Finite Groups} \label{lawssubsect}

We start out with some basic definitions.

\begin{defn}
Fix $x,y$ an ordered basis for the free group $F_2$ 
and let $w \in F_2 \setminus \lbrace 1 \rbrace$. 
For any group $G$ define the \emph{evaluation map} $G \times G \rightarrow G$ 
(also denoted $w$) by 
$w(g,h) = \pi_{(g,h)} (w)$, 
where $\pi_{(g,h)}$ is the (unique) homomorphism 
$F_2 \rightarrow G$ extending $x \mapsto g$, $y \mapsto h$. 
We call $w$ a \emph{law for $G$} if $w(G \times G) = \lbrace 1_G \rbrace$. 
\end{defn}

Of course we could equally define word maps $G^k \rightarrow G$ 
associated to elements of $F_k$ for any $k \geq 1$, 
and thereby seek laws for $G$ within $F_k$, 
however it turns out that very little is lost by restricting to $k=2$. 
For if $k > 2$, standard embeddings of $F_k$ into $F_2$ 
associate to every law $w \in F_k$ for $G$
a law $\tilde{w} \in F_2$ for $G$, 
of length depending linearly on the length of $w$, 
while conversely, an inclusion of a basis for $F_2$ 
into a basis for $F_k$ turns every law for $G$ in $F_2$ 
into a law in $F_k$. 
Meanwhile, a nontrivial element $w \in F_1 \cong \mathbb{Z}$ 
is a law for $G$ iff the exponent of $G$ divides $w$ 
(viewed as an integer). 

We note two basic facts about the structure of laws in finite groups, 
which will enable us to construct new laws from old. 
The first allows us to combine words vanishing on subsets of a group 
to a new word vanishing on the union of those subsets, 
and is proved as Lemma 2.2 in \cite{KozTho}. 
To this end, recall for $G$ a group and $w \in F_2$ 
a word the definition of the \emph{vanishing set} $Z(G,w)$ of $w$ on $G$ from \cite{Thom}: 
\begin{center}
$Z (G,w) = \lbrace (g,h) \in G \times G \mid w(g,h)=1_G \rbrace$. 
\end{center}
\begin{lem} \label{UnionLemma}
Let $w_1 , \ldots , w_m \in F_2$ be non-trivial words. 
Then there exists a non-trivial word $w \in F_2$ 
of length at most $16 m^2 \max_i \lvert w_i \rvert$
such that for all groups $G$, 
\begin{center}
$Z (G,w) \supseteq Z (G,w_1) \cup \ldots \cup Z (G,w_m)$. 
\end{center}

\end{lem}

Note that, as well as allowing us to increase the vanishing set of words within a single group, 
Lemma \ref{UnionLemma} allows us to take a family of groups and, given a law for each group in the family, 
produce a new law which holds in every group in the family simultaneously. 

The second fact allows us to construct laws for group extensions. 
It is proved as Lemma 2.1 in \cite{Thom}. 

\begin{lem} \label{ExtnLemma}
Let $1 \rightarrow N \rightarrow G \rightarrow Q \rightarrow 1$ be an extension of groups. 
Suppose $N, Q$ satisfy non-trivial laws in $F_2$ of length $n_N, n_Q$, respectively. 
Then $G$ satisfies a non-trivial law of length at most $n_N n_Q$. 
\end{lem}

The previous lemma is stated in \cite{Thom} with $n_N (n_Q+2)$ in place of $n_N n_Q$, but the additional summand is easily removed.

\begin{ex} \label{solubleex}
\normalfont A group $A$ is abelian iff the word $x^{-1} y^{-1} x y \in F(x,y)$ 
is a law for $A$. By Lemma \ref{ExtnLemma} and induction, 
it follows that if $G$ is soluble of derived length at most $d$, 
then $G$ satisfies a law of length at most $4 \cdot 6^{d-1}$. 
Since every nilpotent group of class at most $2^d$ 
is soluble of derived length at most $d$, 
these groups also satisfy such a law. 

The work of Elkasapy and the second author \cite{ElkTho} 
on the derived and lower central series of $F_2$ allows one to construct even shorter laws for soluble and nilpotent groups, 
of which we will avail ourselves in the sequel. 
\end{ex}

\subsection{Mixing times and Random Walks} \label{diamsubsect}

Let $G$ be an  finite group, 
and let $S \subseteq G$ be a symmetric generating set. We will consider a lazy random walk associated with the set $S$, as follows: let $x_1 , \ldots , x_l$ be independent random variables, each with distribution function: 
\begin{center}
$\frac{1}{2 \lvert S \rvert} \chi_S + \frac{1}{2} \delta_{1_G}$
\end{center}
where $\chi_S$ is the indicator function of $S$ 
and $\delta_{1_G}$ is the Dirac mass at the identity. 
Let $\omega_l$ be the random variable on $G$ 
given by $\omega_l = x_1 \cdots x_l$. We are interested in the mixing time of the random walk.

Seminal results about the mixing times of random walks on ${\rm PSL}_2(q)$ were derived from diameter bounds proved by Helfgott \cite{Helfgott} (for the case $q$ prime) 
and generalizations (to arbitrary $q$) due to Dinai \cite{Dinai} and Varj\'u \cite{Varju}.
We are going to use the following result due to Breuillard and Gamburd \cite{BreGam} that holds for $\PSL_2(p)$ for a sufficiently large set of primes.

\begin{thm}[Breuillard-Gamburd] \label{bglem}
There is a constant $\delta>0$ such that for all all $n \geq 2$ the number of rational primes $p$ less than $n$ for which $\PSL_2(p)$ has uniform spectral gap less than $\delta$ is at most $n^{1/2}$.
\end{thm}

For sake of convenience, we will refer to the set of primes arising from the previous theorem as the set of good primes. Note that a spectral gap of $\delta$ just means that the spectral radius of the random walk acting on mean-zero functionals is bounded from above by $1 -\delta$ and thus, the random walk approaches equidistribution exponentially fast.
We therefore have the following corollary. 

\begin{coroll} \label{rwcoroll}
Let $G={\rm PSL}_2(p)$ and $p$ a good prime and let $S \subseteq G$ be a generating set 
and let $E \subseteq G$. Then: 
\begin{center}
$\mathbb{P} [\omega_l \in E] \geq \lvert E \rvert / 2 \lvert G \rvert$
\end{center}
for all $l \geq O (\log \lvert G \rvert )$. 
\end{coroll}

For more details on the relevant concepts and results consult \cite{DiaSal}.

\subsection{Residual Finiteness Growth} \label{RFGsubsect}

We recall the definition of the \emph{residual finiteness growth function} 
from the Introduction. Let $\Gamma$ be a finitely generated residually finite group and let $1_{\Gamma} \neq g \in \Gamma$. Define: 
\begin{center}
$k_{\Gamma} (g) = \min \lbrace \lvert Q \rvert \mid \text{there exists }\pi : \Gamma \rightarrow Q, \pi (g) \neq 1_Q \rbrace$. 
\end{center}
Fix a finite generating set $S$ for $\Gamma$. 
For $n \in \mathbb{N}$: 
\begin{center}
$\mathcal{F}_{\Gamma} ^S (n) = \max \lbrace k_{\Gamma} (g) \mid 1_{\Gamma} \neq g \in \Gamma, g \in B_S (n) \rbrace$. 
\end{center}

Here $B_S (n)$ denotes the ball of radius $n$ with respect to the word metric associated with $S$. 
The function $\mathcal{F}_{\Gamma} ^S$, 
known as the \emph{residual finiteness growth function} 
for $\Gamma$ (with respect to $S$) 
was introduced and studied by Bou-Rabee \cite{BouRab}, 
up to a natural notion of equivalence of functions. 

\begin{defn}
For any $f_1 , f_2 : (0,\infty) \rightarrow (0,\infty)$, 
write $f_1 \preceq f_2$ if there exists $C>0$ such that 
$f_1 (x) \leq C f_2 (Cx)$ for all $x \in (0,\infty)$, 
and write $f_1 \simeq f_2$ when both $f_1 \preceq f_2$ 
and $f_2 \preceq f_1$. 
\end{defn}

It is clear from this definition that 
$\simeq$ is an equivalence relation. 
Of course any $f : \mathbb{N} \rightarrow (0,\infty)$ 
can be extended to $(0,\infty)$ via $f (x) = f (\lfloor x \rfloor)$. 
Such functions may thereby also be compared under $\preceq$. 

\begin{lem}[\cite{BouRab}]
Let $H$ be a subgroup of $\Gamma$ generated by a finite set $L$. 
Then $\mathcal{F}_{H} ^L \preceq \mathcal{F}_{\Gamma} ^S$. 
\end{lem}

We conclude the following corollary.

\begin{coroll} \label{lawgensetcoroll}
Let $S,T$ be finite generating sets for $\Gamma$. 
Then $\mathcal{F}_{\Gamma} ^S \simeq \mathcal{F}_{\Gamma} ^T$. 
Thus we may speak without ambiguity about the ($\simeq$-class of the) 
residual finiteness growth function $\mathcal{F}_{\Gamma}$
\emph{of $\Gamma$ itself}.
\end{coroll}

Since any non-abelian finite-rank free group embeds into any other,
we also obtain the following consequence. 

\begin{coroll} \label{compareRFGfreegrps}
Let $\Gamma_1,\Gamma_2$ be finite-rank free groups of rank $\geq 2$. 
Then $\mathcal{F}_{\Gamma_1} \simeq \mathcal{F}_{\Gamma_2}$. 
\end{coroll}

There is a close relationship between residual finiteness growth 
of free groups and laws for finite groups. 

\begin{propn} \label{lawsRFGprop}
Let $\alpha : (0,\infty) \rightarrow (0,\infty)$ 
be a strictly increasing function. 
Suppose that for all $n$ there is a law of length at most $\alpha (n)$ 
simultaneously valid in all groups of order at most $n$. 
Let $\Gamma$ be a finite-rank free group. Then: 
\begin{center}
$\alpha^{-1} \preceq \mathcal{F}_{\Gamma}$. 
\end{center}
\end{propn}

\begin{proof}
By Corollary \ref{compareRFGfreegrps} we may assume $\Gamma = F_2$. 
Fix a basis $S$ for $\Gamma$. 
Let $w \in \Gamma$ be a non-trivial word 
of length at most $\alpha (n)$, 
which is simultaneously a law for all groups of order at most $n$. 
Then $k_{\Gamma} (w) \geq n+1$, 
so $\mathcal{F}_{\Gamma} ^S (\alpha (x)) \geq x$, 
for all $x \in (0,\infty)$. 
The claim now follows from Corollary \ref{lawgensetcoroll}. 
\end{proof}

We may now complete: 

\begin{proof}[Proof of Theorem \ref{RFGthm}]
By Theorem \ref{allgrpsthm} there exists $C>0$ such that we may take 
$\alpha$ as in Proposition \ref{lawsRFGprop} with: 
\begin{center}
$\alpha^{-1} (n) \geq n^{3/2} / C \log(n)^{9/2 + \varepsilon}$. 
\end{center}
\end{proof}

\section{Proof of Theorem \ref{allgrpsthm}} \label{mainproofsect}

\subsection{Reduction to Simple Groups} \label{redsimplesubsect}

As discussed above, 
the proof of Theorem \ref{Thomallgrpsthm} from \cite{Thom} 
begins by reducing the problem of constructing laws valid 
in \emph{all} finite groups up to the order bound 
to that of constructing laws valid only in all 
finite \emph{simple} groups up to the order bound. 
Via the same reductions, our Theorem \ref{allgrpsthm} 
will follow from the following result, 
which improves asymptotically on Proposition 4.1 from \cite{Thom}. 

\begin{propn} \label{allsimplegrpspropn}
For all $n \in \mathbb{N}$, there exists a word $w_n \in F_2$ of length: 
\begin{center}
$O\left( n^{2/3} \log (n) ^3 \right)$
\end{center}
such that for every finite \emph{simple} group $G$ 
satisfying $\lvert G \rvert \leq n$, 
$w_n$ is a law for $G$. 
\end{propn}

Theorem \ref{allgrpsthm} follows from 
substituting Proposition \ref{allsimplegrpspropn} 
in place of Proposition 4.1 from \cite{Thom}, 
and proceeding \emph{mutatis mutandis} 
with the argument as in \cite{Thom}. 
We refer the reader there for the details 
(and in particular for references for many 
of the assertions made in the proof), 
and here restrict ourselves to an outline. 

\begin{proof}[Proof of Theorem \ref{allgrpsthm} (sketch)]
First suppose that $G$ is nilpotent. 
The bound $\vert G \rvert \leq n$ 
implies the nilpotency class of $G$ is at most $\log_2 (n) + 1$ 
(the maximal length of a subgroup chain in $G$). 
By Example \ref{solubleex}, 
$G$ satisfies a law of length:
\begin{center}
$\log (n) ^{O(1)}.$
\end{center}
A more precise analysis using results from \cite{ElkTho} yields a bound of $O(\log(n)^{3/2})$, see Proposition 3.1 in \cite{Thom}.

Now more generally assume $G$ is soluble. 
Let $N$ be the Fitting subgroup of $G$. 
Fix a prime $p$ and let $N(p)$ be the Sylow $p$-subgroup of $N$. 
The action of $\Aut (N(p))$ on the Frattini quotient of $N(p)$ 
(whose rank we denote by $m(p)$)
induces a map $\alpha_p : \Aut (N(p)) \rightarrow \GL_{m(p)} (p)$, 
whose kernel is a $p$-group. 
Moreover the natural map $\psi : G \rightarrow \prod_p \Aut (N(p))$ 
is an embedding modulo the center of $G$. 
Since $\ker (\prod_p \alpha_p) \cap \psi (G)$ is nilpotent and by Lemma \ref{ExtnLemma}, it suffices to find a law for 
$\im ((\prod_p \alpha_p) \circ \psi)$. 

Since $G$ is soluble, so is its image in $\GL_{m(p)} (p)$. 
The derived length of the latter is $O (\log (m(p)))$. 
Note that, since $\vert G \rvert \leq n$, 
$m(p) \leq \log_2 (n)$. 
We apply the bound for laws in soluble groups 
from Example \ref{solubleex} to 
$\im ((\prod_p \alpha_p) \circ \psi)$ 
Putting all this together, $G$ satisfies a law of length:
\begin{center}
$\log (n) ^{O (1)}$. 
\end{center}
Again, a more precise analysis using the results of Elkasapy and the second author on the length of the shortest non-trivial element of in the $k$th step of the derived series $F_2 ^{(k)}$ \cite{ElkTho} yields a bound of $O(\log(n)^{9/2})$, see Proposition 3.2 in \cite{Thom}. Let's fix a constant $D_2>0$, such that there exists a word of length $D_2 \log(n)^{9/2}$, which is satisfied by each solvable group of size at most $n$.

Finally consider a general group $G$. Recall that every finite group is soluble-by-semisimple, 
so by Lemma \ref{ExtnLemma} and the previous paragraph, 
we may assume $G$ is semisimple (in the sense of Fitting). 

There exist finite simple groups $H_i$ and $k_i \in \mathbb{N}$ 
such that $G$ may be identified with a subgroup of 
$\prod_{i=1} ^l G_{(H_i,k_i)}$, for finite groups $G_{(H_i,k_i)}$ 
satisfying: 
\begin{center}
$H_i ^{k_i} \leq G_{(H_i,k_i)} \leq \Aut (H_i) \wr \Sym (k_i)$, 
\end{center}
and in such a way that each $H_i ^{k_i} \leq G$. 
The bound $\lvert G \rvert \leq n$ 
implies upper bounds on $\lvert H_i \rvert$ and $k_i$. 
By Theorem \ref{LawsforAlt} and Lemma \ref{ExtnLemma}, 
the problem is reduced to the construction of short laws for the 
$\Aut (H_i)$. 

For $H$ a finite simple group, the solution to Schreier's Conjecture 
implies that $\Aut (H)/H$ is soluble of derived length at most $3$. 
The result now follows from Proposition \ref{allsimplegrpspropn} 
and final applications of Lemma \ref{ExtnLemma} and Lemma \ref{UnionLemma}. 

In total, having obtained a law of length $O(n^{2/3} \log(n)^3)$, valid for all simple groups up to size $n$, following the arguments in \cite{Thom}, we obtain a constant $D_1>0$ and laws of the length bounded by $D_1 n^{2/3} \log(n)^3$ valid for all semisimple groups of size up to $n$. 

Consider now an increasing sequence of real numbers $a_1,\dots,a_{L+1}$, where $$a_1:=1,  \quad a_{k+1} = \exp\left(a_k^{4/27} \right)$$ and $L$ is the last index, where $a_L \leq n$. It is easy to see that $L= O(\log^*(n))$, where $\log^*$ denotes the iterated logarithm. 

Now, let $G$ be a finite group of size at most $n$ and let $S \lhd G$ be its solvable radical with the associated semisimple quotient $G/S$. It is clear that there must exist some $j \in \{1,\dots, L\}$ such that $|G/S| \leq n/a_j$ and $|S| \leq a_{j+1}.$ Indeed, just take $j$ to be the last index with $|S|\geq a_j$.

Now, in each of the $L$ cases, Lemma \ref{ExtnLemma} yields a law of length bounded by 
$$D_1 (n/a_j)^{2/3} \log(n/a_j)^3 \cdot D_2\log(a_{j+1})^{9/2} \leq D_1D_2 n^{2/3} \log(n)^3,$$
that is valid for those $G$ that fall into this particular case.
Combining all the $L=O(\log^*(n))$ cases using Lemma \ref{UnionLemma}, we obtain a law of length
$$O\left( n^{2/3} \log(n)^3 \log^*(n)^2\right)$$
that is valid for all groups of size at most $n$.
Again, for details we refer to \cite{Thom}.
\end{proof}

\subsection{Reduction to Low-Rank Simple Groups of Lie Type} \label{lowrkreductsubsect}

In fact, the only finite simple groups of Lie type 
for which the laws constructed in the proof of Proposition 4.1 from \cite{Thom} 
do not already satisfy the requirements of Proposition 
\ref{allsimplegrpspropn} are those of the form $\PSL_2 (q)$, 
$\PSL_3 (q)$ and $\PSU_3 (q)$. 
That is, we reduce the proof of Proposition \ref{allsimplegrpspropn} 
to the following statement. 

\begin{propn} \label{smallsimplegrpspropn}
For all $n \in \mathbb{N}$, there exists a word $w_n \in F_2$ of length: 
\begin{center}
$O(n^{2/3})$
\end{center}
such that if $G$ is a finite simple group not of the form 
$\PSL_2 (q)$, $\PSL_3 (q)$ or $\PSU_3 (q)$ for some prime power $q$, 
and $\lvert G \rvert \leq n$, then $w_n$ is a law for $G$. 
\end{propn}

Roughly speaking, the reason these families provided the bottleneck 
for the length of laws in \cite{Thom} was that they are the only 
families in which a finite simple group $G$ may contain an element $g$ 
of large order compared to the order of $G$. 

\begin{proof}[Proof of Proposition \ref{allsimplegrpspropn}]
First assume $G$ to be a finite simple group of Lie type. 
We refer to the tables from \cite{Thom} (p.5), 
which in turn are based on \cite{KanSer}. 
The tables record $a(G),b(G) \in \mathbb{N}$ such that, 
if $G$ is defined over a field of order $q$, 
\begin{center}
$q^{a(G)} \ll \lvert G \rvert \quad \mbox{and} \quad \max_{g \in G} o(g) \ll q^{b(G)}$
\end{center}
(with the implied constants absolute). 
Moreover by inspection of the tables, 
we may take $b(G) \leq 2/9 \cdot a(G)$ in all cases, 
except for when $G$ is of the form 
$\PSL_2 (q)$, $\PSL_3 (q)$ or $\PSU_3 (q)$. 
Excluding the latter possibility, we have: 
\begin{equation} \label{maxeltordereqn}
\max_{g \in G} o(g) \ll \lvert G \rvert ^{2/9} \leq n^{2/9}
\end{equation}
Meanwhile, a classical result of Landau shows that the maximal order 
of an element of $\Alt (k)$ is at most 
$\exp \big(O((k\log(k))^{1/2})\big)$,  
so $G = \Alt (k)$ also satisfies (\ref{maxeltordereqn}). 
Thus there exists an absolute constant $C>0$ such that, 
if $G$ is a finite simple group other than 
$\PSL_2 (q)$, $\PSL_3 (q)$ or $\PSU_3 (q)$, then: 
\begin{center}
$G = \bigcup_{i = 1} ^{Cn^{2/9}} Z(G,x^i)$
\end{center}
Applying Lemma \ref{UnionLemma} to the words $w_i = x^i$ 
for $1 \leq i \leq Cn^{2/9}$, we obtain a law of length $O(n^{2/3})$ 
valid simultaneously in all such $G$. 
Combining (by Lemma \ref{UnionLemma} again) 
this last law with the law obtained in 
Proposition \ref{smallsimplegrpspropn}, 
we obtain the required result. 
\end{proof}

We will conclude by constructing, 
for each of the three families $\PSL_2 (q)$, $\PSL_3 (q)$ 
and $\PSU_3 (q)$, a law of the length $O(n^{2/3} \log(n)^3)$ 
valid in all groups in the family of order at most $n$, this is Proposition \ref{case1} and Proposition \ref{case2}. The proof of Proposition \ref{smallsimplegrpspropn}, 
and hence of Theorem \ref{allgrpsthm}, 
is then completed by combining these three laws using 
Lemma \ref{UnionLemma}. 

\subsection{Short Laws for $\PSL_3 (q)$ \& $\PSU_3 (q)$} \label{PSL3subsect}

The required result for groups of the form 
$\PSL_3 (q)$ and $\PSU_3 (q)$ 
will be obtained by combining the laws 
produced by Theorem \ref{Liethm} 
as $q$ varies over all sufficiently small prime powers. 
To this end, we record the following standard consequence of the 
Prime Number Theorem (see \cite{ErdSur}). 

\begin{lem} \label{NTlem}
For any $n \in \mathbb{N}$, 
the number of prime powers at most $n$ 
is $O(n/\log(n))$, 
and the number of these which are proper powers of primes 
is $O(n^{1/2})$. 

\end{lem}

\begin{propn} \label{case1}
For all $n \in \mathbb{N}$, there exists a word $w_n \in F_2$ of length: 
\begin{center}
$n^{3/8} \log(n)^{O(1)}$
\end{center}
such that if $G$ is equal to $\PSL_3 (q)$ or $\PSU_3 (q)$ for some prime power $q$, 
and $\lvert G \rvert \leq n$, then $w_n$ is a law for $G$. 
\end{propn}

\begin{proof}
First note that $\PSL_3 (q)$ and $\PSU_3 (q)$ 
satisfy the conditions of Theorem \ref{Liethm} 
with $d=3$, so satisfy laws of length $q^{\lfloor 3/2 \rfloor} \log(q)^{O(1)} = q \log(q)^{O(1)}$. 

Alternatively and without using Theorem \ref{Liethm} from \cite{BraTho(Lie)}, we could have studied the orders of elements in $\PSL_3 (q)$ and $\PSU_3 (q)$ more carefully and noted that the order of each element divides $q^2-q, q^2+q, q^2-1, q^2+q+1$ or $q^2-q+1$. Using Lemma \ref{UnionLemma}, this implies immediately that there is a law of length $O(q^2)$ satisfied by these groups -- enough for us to proceed.

Recall that $\PSL_3 (q)$ and $\PSU_3 (q)$ have order 
proportional to $q^8$, so if $G$ is isomorphic to some 
$\PSL_3 (q)$ or $\PSU_3 (q)$ and $\lvert G \rvert \leq n$, 
then $q = O(n^{1/8})$. By Lemma \ref{NTlem}, 
there are $O(n^{1/8}/\log(n))$ such prime powers $q$. 

We may thus apply Lemma \ref{UnionLemma} with 
$m = O(n^{1/8}/\log(n))$ and $w_i$ of length 
$n^{1/8} \log(n)^{O(1)}$ (or merely $O(n^{1/4})$ using the alternative argument)
to obtain a law of length $n^{3/8} \log(n)^{O(1)}$ (or just $O(n^{1/2}/\log(n)^2)$ on the alternative route)
valid in all groups of the required form. 
\end{proof}

\subsection{Short Laws for $\PSL_2 (q)$} \label{PSL2subsect}

Observe first that the argument of the previous subsection 
necessarily fails for $\PSL_2 (q)$. 
For let $u_q$ be a law for $\PSL_2 (q)$. 
Then $u_q$ has length $\Omega (q)$ (see \cite{Had}). 
Combining the laws $u_q$ by Lemma \ref{UnionLemma} 
over all values of $q$ such that 
$\lvert \PSL_2 (q) \rvert \leq n$ (that is, for $q = O(n^{1/3})$) 
would yield a law of length $\Omega (n/ \log(n)^2)$, 
which is unacceptable for our purposes. However, this general strategy will be the way to treat prime powers and primes in the set of bad primes arising from Theorem \ref{bglem} 
(recall that \emph{good} and \emph{bad primes} 
were defined in Subsection \ref{diamsubsect}). 

\begin{lem} \label{acceptable}
For all $n \in \mathbb{N}$ 
there exists a word $w_n \in F_2$ of length bounded by
$O(n^{2/3})$
such that for $\PSL_2(q)$ with $q$ is either a proper prime power or a bad prime and satisfying $\lvert \PSL_2(q) \rvert \leq n$, 
$w_n$ is a law for $\PSL_2(q)$.  
\end{lem}
\begin{proof} The size of $\PSL_2(q)$ is about $q^3$ so that we have to consider prime powers up to size $O(n^{1/3})$. The size of the set of proper prime powers and bad primes below $O(n^{1/3})$ is bounded by $O(n^{1/6})$, see Lemma \ref{NTlem}. 
It is well-known that laws of length $O(n^{1/3})$ for $\PSL_2 (q)$ exist (see the proof of Proposition 4.1. in \cite{Thom}). Combining all these laws yields the desired law using Lemma \ref{UnionLemma} of length $O(n^{2/3})$.
\end{proof}

We are now left to deal with groups $\PSL_2(q)$ where $q$ is a good prime.
We divide the problem and seek separately short words 
whose associated vanishing sets contain, respectively, 
generating and non-generating pairs of elements in groups 
of the form $\PSL_2 (q)$. 
The strategy in both cases closely parallels that 
employed in the proofs of Theorems \ref{LawsforAlt} and \ref{Liethm}. 

For generating pairs we will use upper bounds 
on the mixing time of $\PSL_2 (q)$ arising from Theorem \ref{bglem}.  
This allows us to give a probabalistic construction of pairs of words, 
whose evaluation maps have images contained in a common soluble subgroup of $\PSL_2 (q)$ (and hence satisfy a short relation). 
For non-generating pairs we use the classification of subgroups of $\PSL_2 (q)$. 

We start by recording an elementary observation from linear algebra. 

\begin{lem} \label{SL2diaglem}
The number of elements of $\SL_2 (q)$ 
which are diagonalisable over $\mathbb{F}_q$ 
is $\Omega (q^3)$. 
\end{lem}

Since the order of $\SL_2 (q)$ is proportional to $q^3$, 
Lemma \ref{SL2diaglem} says precisely that 
an absolutely positive proportion of elements 
are diagonalisable over $\mathbb{F}_q$. 
Now consider the subgroup $U(q) \leq \SL (q)$ of 
upper-triangular elements.
$U(q)$ contains the diagonal subgroup of $\SL_2 (q)$, 
so by Lemma \ref{SL2diaglem}, 
an absolutely positive proportion of the elements of $\SL_2 (q)$ 
are conjugate into $U(q)$. 

By a \emph{Borel subgroup} of $\PSL_2 (q)$ 
we shall mean the image, under the natural projection $\SL_2 (q) \rightarrow \PSL_2 (q)$, of a conjugate in $\SL_2 (q)$ of $U(q)$. 
This is not really the ``right'' way to define these subgroups, 
but it is the most convenient for our purposes. 
The only facts we require about Borel subgroups in the sequel are: 
\begin{itemize}
\item[$(a)$] Every Borel subgroup is metabelian; 
\item[$(b)$] Every Borel subgroup has index $\ll q$ in $\PSL_2 (q)$; 
\item[$(c)$] A positive proportion of the elements of $\PSL_2 (q)$ 
lie in a Borel subgroup
\end{itemize}
Indeed, $(a)$ and $(b)$ are well-known and $(c)$ is just the summary of the preceding discussion. 

We shall also require some facts about free groups. 
The first of these is standard. 

\begin{lem}[see for instance \cite{John} Chapter 1] \label{2genfreesubgrplem}
Let $a,b \in F_2$. 
Then either \begin{itemize} \item[$(i)$] $a$ and $b$ commute, 
and are powers of a common element of $F_2$, 
or \item[$(ii)$] the subgroup $\langle a,b \rangle$ of $F_2$ 
generated by $a$ and $b$ is isomorphic to $F_2$, 
and $\lbrace a,b \rbrace$ is a free generating set.
\end{itemize}
\end{lem}

We further recall Kesten's result 
on the exponential decay of simple random walks on $F_2$. 

\begin{thm}[\cite{Kest}] \label{Kesten}
There exists a constant $\alpha > 0$ such that, 
for any $g \in F_2$, if $w_l$ is the result of a simple random walk 
of length $l$ on a free generating set for $F_2$, 
\begin{center}
$\mathbb{P} [w_l = g] \ll \exp (-\alpha l)$. 
\end{center}
\end{thm}

The key consequence of these two facts that we shall use is the following, which tells us that with high probability, 
the outcomes of a pair of random walks on $F_2$ 
fall into case (ii) of Lemma \ref{2genfreesubgrplem}, see also \cite[Lemma 2.1]{BreGam}.

\begin{coroll} \label{commpairsinfreegrplem}
Let $\alpha > 0$ be as in Theorem \ref{Kesten}. 
Let $u_l , v_l$ be the results of two independent 
simple random walks of length $l$ on a free generating set for $F_2$. 
Then: 
\begin{center}
$\mathbb{P} \big[ [u_l , v_l]=1 \big] \ll l \exp (-\alpha l)$. 
\end{center}
\end{coroll}

\begin{proof}
It is a basic fact that $v=w^k$ in $F_2$ for some non-trivial $w$ implies that the word length of $v$ is at least $k$. Moreover, it follows from Lemma \ref{2genfreesubgrplem} that each non-trivial $u \in F_2$ lies in a unique maximal abelian subgroup, which is isomorphic to $\mathbb Z$ -- generated by some maximal root $w$ of $u$. Thus, for any fixed $u_l$ of length at most $l$, there is a maximal root $w_l$ and $[u_l,v_l]=1$ with $v_l$ of length at most $l$ implies $v_l=w_l^k$ for some $k \in \mathbb N$ satisfying $-l \leq k \leq l$ by Lemma \ref{2genfreesubgrplem}. Hence, as long as $u_l$ is non-trivial, the probability that $v_l$ satisfies $[u_l,v_l]=1$ is bounded by $O(l \exp (-\alpha l))$. However, the probability that $u_l$ is trivial is bounded by $O( \exp (-\alpha l))$. This proves the claim.
\end{proof}

Finally, we record a taxonomy of the subgroups of $\PSL_2 (q)$, 
which follows from the classical results of Dickson \cite{Dick}.

\begin{thm}[Dickson] \label{SL2subgrpclass} Let $q$ be a prime and let $H$ be a proper subgroup of $\PSL_2 (q)$. 
Then one of the following holds. 
\begin{itemize}
\item[$(i)$] $H$ is metabelian; 
\item[$(ii)$] $\lvert H \rvert \leq 60$.
\end{itemize}
\end{thm}

We have arrived at the heart of the proof -- our new approach to treat the crucial case $\PSL_2(q)$.

\begin{propn} \label{case2}
For all $n \in \mathbb{N}$, there exists a word $w_n \in F_2$ of length: 
\begin{center}
$O\left(n^{2/3} \log(n)^{3} \right)$
\end{center}
such that if $G$ is equal to $\PSL_2 (q)$ for some prime power $q$, 
and $\lvert G \rvert \leq n$, then $w_n$ is a law for $G$. 
\end{propn}
\begin{proof}
Let $u_1 , \ldots , u_m , v_1 , \ldots , v_m$ 
be the results of $2 m$ independent lazy random walks of length 
$l = C_1\log (n)$ on a free generating set for $F_2$, 
where $C_1$ is a sufficiently large absolute constant. 
Fix (for the time being) a good prime $q$ such that 
$\lvert \PSL_2 (q) \rvert \leq n$ 
and a generating pair $g,h \in \PSL_2 (q)$. 

For each $1 \leq i \leq m$, 
the probability that $u_i (g,h)$ 
lies in a Borel subgroup
is at least a positive absolute constant $C_2$, 
by Lemma \ref{SL2diaglem} and Corollary \ref{rwcoroll} 
(since we are assuming $C_1$ is sufficiently large, 
Corollary \ref{rwcoroll} is indeed applicable). 

Suppose $u_i (g,h)$ 
does indeed lie in a Borel subgroup $B (q) \leq \PSL_2 (q)$. 
By independence of $u_i$ and $v_i$, 
and applying Corollary \ref{rwcoroll} again, 
the probability that $v_i (g,h)$ lies in $B(q)$ is at least $1/C_3 q$, 
for an absolute constant $C_3 > 0$. 
Thus the probability that $u_i (g,h)$ and $v_i (g,h)$ 
do not lie in a common Borel subgroup is at most 
\begin{center}
$1-C_2/C_3 q \leq 1-C_2/C_3 n^{1/3}$. 
\end{center}
By independence of the $u_j , v_j$, 
the probability that for every $1 \leq j \leq m$ 
the pair $(u_j (g,h),v_j (g,h))$ 
fail to lie in a common Borel subgroup 
is at most $(1-C_2/C_3 n^{1/3})^m$. 
Setting $m = C_4 n^{1/3} \log (n)$, 
for $C_4$ a sufficiently large constant, 
we have 
\begin{center}
$(1-C_2/C_3 n^{1/3})^m \leq \exp (-C_5 \log (n)) = n^{-C_5}$ 
\end{center}
for $C_5$ a constant, which we may take to be arbitrarily large. 

The number of possible generating pairs $(g,h)$ for $\PSL_2 (q)$ 
is no more than $\lvert \PSL_2 (q) \rvert^2 \leq n^2$, 
while the number of possibilities for $q$ is bounded by $O (n^{1/3} / \log (n)),$ 
using Lemma \ref{NTlem}. 
Thus, taking a union bound over all possible good primes $q$ and $(g,h)$, 
the probability of the event: ``for every $1 \leq j \leq m$ 
there exists a generating pair $g,h$ for some $\PSL_2 (q)$ of order at most $n$ such that the pair $(u_j (g,h),v_j (g,h))$ 
fails to lie in a common Borel subgroup'' 
is at most: 
\begin{equation} \label{borelboundeqn}
O \left(n^{7/3 - C_5} / \log (n) \right)
\end{equation}
We will choose $C_5 > 7/3$. 

Meanwhile, by Corollary \ref{commpairsinfreegrplem}, 
for fixed $i$ the probability that $u_i$ and $v_i$ commute (in $F_2$) 
is $O(l \exp (-\alpha l))$, 
so taking a union bound, 
the probability that there exists $1 \leq i \leq m$ 
such that $u_i$ and $v_i$ commute is: 
\begin{equation} \label{commboundeqn}
O (m l \exp (-\alpha l)) 
= O \left(n^{1/3-\alpha C_1} \log (n)^2\right)
\end{equation}
(of course, we could get a better bound by using the independence of the pairs $(u_i,v_i)$, but (\ref{commboundeqn}) will prove more than adequate for our current needs). 

Combining (\ref{borelboundeqn}) and (\ref{commboundeqn}), 
the probability of the event ``either there exists $1 \leq i \leq m$ 
such that $u_i$ and $v_i$ commute or for every $1 \leq i \leq m$ 
there exists a generating pair $g,h$ for some $\PSL_2 (q)$ 
of order at most $n$ such that the pair $(u_j (g,h),v_j (g,h))$ 
fails to lie in a common Borel subgroup'' is at most: 
\begin{center}
$O \left(n^{7/3 - C_5} / \log (n) 
+ n^{1/3-\alpha C_1} \log (n)^2 \right)
< 1$
\end{center}
for $C_1 > 1/\alpha$, $C_5 > 7/3$ and $n$ larger than an absolute constant. 

Therefore there exist \emph{deterministically} words 
$u_1 , \ldots , u_m , v_1 , \ldots , v_m \in F_2$ 
with $m = O (n^{1/3} \log (n))$ 
of length at most $l=O(\log (n))$ 
such that no pair $u_i,v_i$ commutes in $F_2$, 
and such that for every $\PSL_2 (q)$ of order at most $n$, where $q$ is a good prime,  
and every generating pair $g,h \in \PSL_2 (q)$, 
there exists $1 \leq i \leq m$ for which $(u_i (g,h),v_i (g,h))$ 
lie in a common Borel subgroup. 

Consider the word $\tilde{w} = [[a,b],[b,a^{-1}]] \in F(a,b)$. 
It is easy to see that $\tilde{w}$ is a non-trivial element of 
$F(a,b)^{(2)}$ of length $14$. 
Thus, on the one hand $w_i = \tilde{w}(u_i,v_i) \in F_2$ 
is non-trivial for every $1 \leq i \leq m$, since by Lemma \ref{2genfreesubgrplem} $u_i,v_i$ 
freely generate a group isomorphic to $F(a,b)$. 
On the other hand, since every Borel subgroup of $\PSL_2 (q)$ 
is metabelian, for every $\PSL_2 (q)$ of order at most $n$, where $q$ is a good prime,  
and every generating pair $g,h \in \PSL_2 (q)$, 
there exists $1 \leq i \leq m$ for which 
$$w_i (g,h) = \tilde{w}(u_i(g,h),v_i(g,h)) = 1.$$ 
Applying Lemma \ref{UnionLemma} to the words $w_1 , \ldots , w_m$, 
we obtain a word $w_{gen} \in F_2$ of length bounded by
$$16(C_4 n^{1/3} \log(n))^2 \cdot (14\cdot C_1\log(n)) = O(n^{2/3} \log (n)^3),$$ 
which vanishes at all such generating pairs $g,h$. 

Finally we produce a word of bounded length whose vanishing set 
contains all non-generating pairs. 
This will be achieved using Theorem \ref{SL2subgrpclass}. 
Indeed, if $g,h$ generate a metabelian subgroup or a group of bounded order 
(conclusions (i) or (ii) of Theorem \ref{SL2subgrpclass}), 
then $(g,h)$ lies in the vanishing set of a word of bounded length $w_{sub}$. 
Applying Lemma \ref{UnionLemma} one more time, to $w_{gen}, w_{sub}$ and the law obtained from Lemma \ref{acceptable} to cover also the case when $q$ is a proper prime power or a bad prime, we have the required result. 
\end{proof}

\section{Open Problems}

We end this article with various possibly approachable open problems.
At the moment, our constructions (or rather proof of existence) of laws are inherently random and say only very little about the shape of these elements in the free group.

\begin{prob}
Give an explicit construction of short laws holding in:  
\begin{itemize}
\item[$(a)$] symmetric groups; 

\item[$(b)$] finite simple groups of Lie type; 

\item[$(c)$] all groups of order at most $n$ 
(improving on Theorem \ref{Thomallgrpsthm}). 
\end{itemize}
\end{prob}

It remains plausible that there could exist laws of polynomial length for ${\rm Sym}(n)$. Note that assuming Babai's Conjecture saying that $\diam(\Sym(n))$ should be $n^{O(1)}$ (see \cite{Bab}), Kozma and the second named author showed in \cite{KozTho} that laws of ${\rm Sym}(n)$ of length $n^{O(\log\log(n))}$ exist. Maybe a refined argument could also prove a polynomial bound assuming Babai's Conjecture. Based on existing polynomial diameter estimates for random generators of ${\rm Sym}(n)$ \cite{HelfZuk}, words in $F_2$ of length $n^8 \log(n)^{O(1)}$ can be constructed that are laws for {\it almost all} of ${\rm Sym}(n)$, see \cite{Zyrus}. Following \cite{Zyrus}, there exist similar improved bounds for {\it almost laws} of finite simple groups of Lie type. Indeed, for example the group $\SL_d(q)$ satisfies an almost law of length $q \log(q)^{O_d(1)}$. This will be explained in more detail in \cite{BraTho(Lie), Zyrus}.

\begin{prob}
Give new upper bounds on the length of the shortest law for ${\rm Sym}(n)$. Can there be laws of length bounded by $n^{O(1)}$?
\end{prob}

On the other side, lower bounds are equally interesting.

\begin{prob} \label{lbdlawSym}
Give new lower bounds on the length of the shortest law for 
$\Sym (n)$, say of the form $\Omega (n \log (n))$ or $\Omega(n^2)$. 
\end{prob}

Note that $\PSL_2(q) \subseteq \Sym(n)$ for $q<n$, so that a sharp complementary bound in Theorem \ref{allgrpsthm} following the strategy by Kassabov-Matucci (see (\cite[Remark 9]{KasMat} and the remarks after Theorem \ref{RFGthm}), would potentially also prove $\Omega(n^2)$ with respect to the previous problem. As regards a possible result complementary to Theorem \ref{allgrpsthm}, even a bound of 
$\Omega(n^{1/3} \log(n))$ would be an improvement 
to the state of the art. 

\begin{prob}
Give an improved upper bound on $\mathcal{F}_{F_k} ^S (n)$, smaller than $O(n^3)$, for $k \geq 2$. 
\end{prob}

A strengthening of Problem \ref{lbdlawSym} is the following, 
which already seems very achievable. 

\begin{prob} \label{girth}
Find a generating set $X_n$ of $\Sym (n)$ 
such that the associated Cayley graph has girth $\Omega (n \log (n))$. 
\end{prob}

Again, the best construction of generating sets of ${\rm Sym}(n)$ with respect to Problem \ref{girth} is random in nature and yields a bound on the girth of $\Omega((n \log(n))^{1/2})$, see \cite[Thm.\ 3]{virag}. We are not aware of an explicit family of sets of generators satisfying this lower bound on the girth. The authors of \cite{virag} conjecture a positive answer to Problem \ref{girth} for random generators.

\vspace{0.1cm}

Needless to say, the exact determination the of the length and shape of shortest laws for symmetric groups or all finite group up to a certain size remains an outstanding open problem in finite group theory.

\section*{Acknowledgements}

The first named author wishes to thank the Institut f\"{u}r Geometrie 
of the Technische Universit\"{a}t Dresden 
for providing him with a hospitable welcome 
on the occasion of his visit in June 2016, 
during which the essence of the proof of Theorem \ref{allgrpsthm} 
was first distilled. This research was supported by ERC Starting Grant No.\ 277728 and ERC Consolidator Grant No.\ 681207.

We thank the unknown referee for interesting comments that led us to include a refined version of our main result.

\Addresses


\begin{thebibliography}{99}

\bibitem{Bab}
L. Babai and  \'A. Seress.
{\it On the diameter of Cayley graphs of the symmetric group.}
J. Combinatorial Theory-A 49 (1988), 175-179.

\bibitem{BouRab} K. Bou-Rabee. 
{\it Quantifying residual finiteness.}
J. Algebra 323 (2010), 729-737

\bibitem{BouMcR} K. Bou-Rabee and D.B. McReynolds. 
{\it Asymptotic growth and least common multiples in groups.}
Bull. Lond. Math. Soc. 43 (2011) no. 6, 1059-1068

\bibitem{BraTho(Lie)} H. Bradford and A. Thom. 
{\it Short laws for finite groups of Lie type.} 
In preparation. 

\bibitem{BreGam}
E. Breuillard and A. Gamburd.
{\it Strong uniform expansion in SL(2,p).}
Geom. Funct. Anal. 20 (2010), no. 5, 1201--1209.

\bibitem{DiaSal} P. Diaconis and L. Saloff-Coste. 
{\it Comparison techniques for random walk on finite groups.}
Ann. Probab. 21, Issue 4 (1993), 2131-2156

\bibitem{Dick} L.E. Dickson. 
{\it Linear groups, with an Exposition of the Galois Field Theory. }
Teubner, Leipzig, 1901.

\bibitem{Dinai} O. Dinai. 
{\it Growth in $\SL_2$ over finite fields.}
J. Group Theory 14, Issue 2 (2011), 273--297

\bibitem{ElkTho} A. Elkasapy and A. Thom. 
{\it On the length of the shortest non-trivial element in the derived and the lower central series.}
J. Group Theory 18, Issue 5 (2015), 793-804

\bibitem{ErdSur} P. Erd\H{o}s and J. Sur\'{a}nyi. 
{\it Topics in the Theory of Numbers.}
Springer Science and Business Media, New York (2003)

\bibitem{virag} A. Gamburd, S. Hoory, M. Shahshahani, A. Shalev, and B. Vir\'ag.
{\it On the girth of random Cayley graphs.}
Random Structures Algorithms 35 (2009), no. 1, 100--117. 

\bibitem{Had} U. Hadad. 
{\it On the shortest identity in finite simple groups of Lie type.}
J. Group Theory 14 (2011) no. 1, 37-47

\bibitem{Helfgott} H. Helfgott. 
{\it Growth and Generation in $\SL_2 (\mathbb{Z}/p\mathbb{Z})$.}
Ann. Math. 167 (2008) 601--623

\bibitem{HelfZuk}
H. Helfgott, \'A. Seress, and A. Zuk.
{\it Random generators of the symmetric group: diameter, mixing time and spectral gap.}
J. Algebra 421 (2015), 349--368.

\bibitem{John} D.L. Johnson. 
{\it Presentations of Groups.} 
Cambridge University Press (1997)

\bibitem{KanSer} W. Kantor and A. Seress. 
{\it Large element orders and the characteristic of Lie-type simple groups.} 
J. Algebra 322 (2009), 802-832

\bibitem{KasMat} M. Kassabov and F. Matucci. 
{\it Bounding the residual finiteness of free groups.} 
Proc. Amer. Math. Soc. 139 (2011) no. 7, 2281-2286

\bibitem{Kest} H. Kesten. 
{\it Symmetric random walks on groups.} 
Trans. Amer. Math. Soc. 92 (1959), 
Issue 2, 336-354. 

\bibitem{King} O.H. King. 
{\it The subgroup structure of finite classical groups in terms of geometric configurations.} 
Surveys in combinatorics, 2005. Edited by B. S. Webb. 

\bibitem{KozTho} G. Kozma and A. Thom. 
{\it Divisibility and laws in finite simple groups.} 
Math. Ann. 361, Issue 1 (2016), 79-95

\bibitem{BNeum} B.H. Neumann. 
{\it Identical Laws in Groups I.} 
Math. Ann. 114, Issue 1 (1937), 506-525

\bibitem{HNeum} H. Neumann. 
{\it Varieties of Groups.} 
Springer Berlin Heidelberg (1967)

\bibitem{Riv} I. Rivin. 
{\it Geodesics with one self-intersection, and other stories.} 
 Adv. Math. 231 (2012), no. 5, 2391--2412.

\bibitem{Thom} A. Thom. 
{\it About the length of laws for finite groups.} 
Isr. J. Math. (2017) 219, Issue 1, 469--478.

\bibitem{Varju} P.P. Varj\'u. 
{\it Expansion in $\SL_d (\mathcal{O}_K/I)$, I square-free.}
J. Eur. Math. Soc. 14 (2012), no. 1, 273--305

\bibitem{Zyrus} C. Zyrus.
{\it Almost laws for finite simple groups.}
PhD Thesis, in preparation.
\end{thebibliography}
\end{document}